\documentclass[reqno]{amsart}
\usepackage[margin=2.5cm]{geometry}
\usepackage[utf8]{inputenc}
\usepackage[T1]{fontenc}

\usepackage{amsmath,amsthm,amssymb,amsfonts}
\usepackage{graphicx,subfig}
\usepackage[english]{babel}
\usepackage{xcolor}
\usepackage{hyperref}

\usepackage{mathtools}
\usepackage{esint}

\newcommand{\R}{\mathbb{R}}
\newcommand{\N}{\mathbb{N}}

\newcommand{\B}{\mathbb{B}}

\newcommand{\s}{\mathbb{S}}
\newcommand{\eps}{\varepsilon}

\theoremstyle{definition}
\newtheorem{definition}{Definition}[section]
\newtheorem{exemple}[definition]{Example}

\theoremstyle{plain}
\newtheorem{theoreme}[definition]{Theorem}
 \newtheorem{proposition}[definition]{Proposition}
\newtheorem{lemme}[definition]{Lemma}
\newtheorem{corollaire}[definition]{Corollary}

\renewcommand\footnotemark{}

\begin{document}
\title{Multiple tubular excisions and large Steklov eigenvalues}

\author{Jade Brisson} 
\thanks{ D\'epartement de math\'ematiques et de statistique, Pavillon Alexandre-Vachon, Universit\'e Laval, Qu\'ebec, QC, G1V 0A6, Canada; \href{mailto:jade.brisson.1@ulaval.ca}{\nolinkurl{jade.brisson.1@ulaval.ca}}}
\date{}


\begin{abstract}
 Given a closed Riemannian manifold $M$ and $b\geq2$ closed connected submanifolds $N_j\subset M$ of codimension at least $2$, we prove that the first non-zero eigenvalue of the domain $\Omega_\varepsilon\subset M$ obtained by removing the tubular neighbourhood of size $\varepsilon$ around each $N_j$ tends to infinity as $\varepsilon$ tends to $0$.
More precisely, we prove a lower bound in terms of $\varepsilon$, $b$, the geometry of $M$ and the codimensions and the volumes of the submanifolds and an upper bound in terms of $\varepsilon$ and the codimensions of the submanifolds. For eigenvalues of index $k=b\,,b+1\,,\ldots$, we have a stronger result: their order of divergence is $\varepsilon^{-1}$ and their rate of divergence is only depending on $m$ and on the codimensions of the submanifolds.
\end{abstract}

\maketitle

\section{Introduction}

Let $(\Omega,g)$ be a smooth compact connected Riemannian manifold with boundary of dimension $m\geq2$.
A real number $\sigma\in\R$ is called a Steklov eigenvalue if there exists a non-zero function $f\in C^\infty(\Omega)$ such that
\begin{equation*}
\begin{cases}
    \Delta f=0 &\mbox{ in }\Omega,\\
    \partial_n f=\sigma f &\mbox{ on }\partial\Omega.
\end{cases}
\end{equation*}
Here and elsewhere $\Delta=\Delta_g:C^\infty(\Omega)\to C^\infty(\Omega)$ is the Laplace operator induced by the Riemannian metric $g$, and $\partial_n$ denotes the outward-pointing normal derivative on $\partial\Omega$.  It is well known that the Steklov eigenvalues of $\Omega$ form an increasing sequence of positive numbers
\[0=\sigma_0(\Omega,g)<\sigma_1(\Omega,g)\leq\sigma_2(\Omega,g)\leq\ldots\nearrow +\infty\,,\]
where each eigenvalue is repeated according to its multiplicity. This sequence is known as the Steklov spectrum of $(\Omega,g)$. See \cite{GP2017,colbois2022recent} and references therein for more information on this problem.

The Steklov eigenvalues are given by the following variational characterization.
Let $E_k$ be the set of $(k+1)-$dimensional linear subspaces of the Sobolev space $W^{1,2}(\Omega)$.
For $f\in E_k\backslash\{0\}$, define its Rayleigh--Steklov quotient by
\[R_g(f):=\frac{\int\limits_\Omega |\nabla f|^2\,dv_g}{\int\limits_{\partial\Omega}f^2\,dS_g}\,,\]
where $dv_g$ is the volume measure induced by the metric $g$ and $dS_g$ is the boundary measure induced by the metric $g$. Its numerator is called the Dirichlet energy of the function $f$.
Then, for each $k\geq0$, we have
\begin{equation}\label{eq:varchar1}
\sigma_k(\Omega,g)=\min\limits_{E\in E_k}\max\limits_{0\neq f\in E}R_g(f)\,.
\end{equation}
For $k=1$, the characterization \eqref{eq:varchar1} is equivalent to
\begin{equation}\label{eq:varchar2}
\sigma_1=\min\bigg\{\int\limits_\Omega |\nabla f|^2\,dv_g~:~f\in W^{1,2}(\Omega)\,,\int\limits_{\partial\Omega}f\,dS_g=0\,,\int\limits_{\partial\Omega}f^2\,dS_g=1\bigg\}\,.
\end{equation}

\subsection{Tubular excision}
Let $(M,g)$ be a closed and compact smooth Riemannian manifold. Given $N\subset M$ a closed connected submanifold of dimension $0\leq n\leq m-2$ and $\varepsilon>0$, consider the tubular neighbourhood $N^\varepsilon$ of size $\varepsilon$ around $N$ defined by $N^\varepsilon=\{x\in M~|~\text{dist}(x,N)\leq\varepsilon\}$. 
Here and throughout the paper, $\text{dist}$ is the distance induced by the Riemannian metric $g$. 
In \cite{brisson}, the Steklov problem on the domain $\Omega_\varepsilon:=M\backslash N^\varepsilon$ obtained by the excision of the tubular neighbourhood is studied under the crucial assumption that $N$ is connected.  
This provides new examples of domains $\Omega_\varepsilon\subset M$ with arbitrarily large Steklov eigenvalues. In fact, in \cite[Theorem 1.1]{brisson}, it is proved that, for every $k\geq1$, $\varepsilon\sigma_k(\Omega_\varepsilon,g)$ converges as $\varepsilon\to0$ to a limit that only depends on the dimensions of $M$ and $N$. The proof is based on the following main arguments. 
We first noted that a neighbourhood $N^\delta\backslash N^\varepsilon$ of the boundary $\partial\Omega_\varepsilon$ is quasi-isometric to the product manifold $(\varepsilon,\delta)\times N\times\s^{m-n-1}$. Then, we used the classical Dirichlet--Neumann bracketing method in order to bound the Steklov eigenvalues of $\Omega_\varepsilon$ between eigenvalues of mixed problems (Steklov--Neumann and Steklov--Dirichlet problems) on the product manifold. Those eigenvalues were then computed by separation of variables. With this method, the eigenvalues are naturally indexed by two indices $k\,,p$, where the index $k$ refers to the Laplace eigenvalues of the submanifold $N$ and the index $p$ refers to the Laplace eigenvalues of the sphere $\s^{m-n-1}$ (see Lemma \ref{lemme:asymp} below).
This argument can be adapted on a domain $\Omega_\varepsilon$ obtained by the excision of $b\geq2$ tubular neighbourhoods around connected submanifolds $N_j\subset M$ of positive condimension to give the rate of divergence of $\sigma_k(\Omega_\varepsilon)$ for $k\geq b$. This result is described in the next theorem. Recall that the distinct Laplace eigenvalues on the $(m-n-1)-$dimensional sphere are given by $\lambda_{(i)}=i(i+m-n-2)$. The multiplicity of $\lambda_{(i)}$ is $m_i=\binom{m-n-1+i}{m-n-1}-\binom{m-n+i-3}{m-n-1}$. Throughout the paper, the notion of truncated Steklov spectrum, as defined below, is used.

\begin{definition}\label{def:truncatedss}
The truncated Steklov spectrum is the following set of eigenvalues:
\[\{\sigma_\ell(\Omega)~:~\ell\geq b\}\,.\]
\end{definition}

\begin{theoreme}\label{thm:convergence}
Let $M$ be a closed Riemannian manifold of dimension $m\geq 2$. Let $N_1\,,\ldots\,,N_b\subset M$ be $b\geq2$ disjoint closed connected submanifolds such that the dimension of $N_j$ is $0\leq n_j\leq m-2$. For $\varepsilon>0$ small enough, the truncated Steklov spectrum of $\Omega_\varepsilon:=M\backslash\bigsqcup\limits_{j=1}^bN_j^\varepsilon$ is given by a disjoint union of $b$ families $\mathcal{S}(N_j)$ of eigenvalues 
\[\{\sigma_\ell(\Omega_\varepsilon)~:~\ell\geq b\}=\bigsqcup\limits_{j=1}^b \mathcal{S}(N_j)\,.\]
Each family $\mathcal{S}(N_j)$ is the contribution of a submanifold $N_j$ and is given by
\[\mathcal{S}(N_j)=\bigg\{\sigma_{k,p}^j(\varepsilon)~:~\{k,p\}\in\N^2\setminus\{(0,0)\}\bigg\}\,,\]
if $N_j$ is not a point, or by
\[\mathcal{S}(N_j)=\{\sigma_k^j(\varepsilon)~:~k>0\}\,,\]
if $N_j$ is a point.
The eigenvalues in each family have the following behaviour.
\begin{enumerate}
\item  Suppose that $N_j$ is not a point. Let $(k,p)\in\N^2\setminus\{(0,0)\}$. For $p=0$, set $q=0$ and for $p>0$, let $q>0$ be such that $m_0+\cdots+m_{q-1}\leq p< m_0+\cdots+m_{q-1}+m_{q}$. The following limit holds
\begin{gather}\label{eq:convergence}
  \lim\limits_{\varepsilon\to0}\varepsilon\sigma_{k,p}^j(\varepsilon)=m-n_j-2+q\,.
\end{gather}
In particular, if $n_j=m-2$ and $m>2$, then for $p=0$, this limit is 0. In that case, the following improvement holds for each $k>0$,
\begin{gather}\label{eq:logimprovement}
  \lim\limits_{\varepsilon\to0}\varepsilon|\log\varepsilon|\sigma_{k,0}^j(\varepsilon)=1\,.
\end{gather}
\item Suppose that $N_j$ is a point. Let $k>0$. Choose the unique $q>0$ such that $m_0+\cdots+m_{q-1}\leq k< m_0+\cdots+m_{q-1}+m_{q}$. The following limit holds
\begin{gather}\label{eq:convergencepoint}
  \lim\limits_{\varepsilon\to0}\varepsilon\sigma_k^j(\varepsilon)=m-2+q\,.
\end{gather}
\end{enumerate}
\end{theoreme}

In the limit $\eps\to 0$, the ordered eigenvalues $\sigma_\ell(\Omega_\eps)$ for $\ell\geq b$ correspond to the smallest cluster at $n=\max n_j$ and at $q=0$ if $n\neq 0$ or at $q=1$ if $n=0$.
\begin{corollaire}\label{cor:introsigmak}
 Let $M$ be a smooth compact Riemannian manifold of dimension $m\geq 2$ and let $N_1\,,\ldots\,,N_b\subset M$ be $b$ smooth closed connected embedded submanifolds of dimension $0\leq n_j\leq m-2$.
Then for each $\ell\geq b$, if $n=\max n_j$, 
  \begin{gather}\label{eq:limitsigmakintro}
    \lim\limits_{\eps\to0}\eps\sigma_{\ell}(\Omega_\varepsilon)=m-n-2+q\,,
  \end{gather}
where $q=0$ if $n>0$ and $q=1$ if $n=0$.
Moreover, in the case where $n=m-2\neq0$, the following holds for each $\ell\in\N$,
  \[\lim\limits_{\eps\to0}\eps|\log\varepsilon|\sigma_{\ell}(\Omega_\varepsilon)=1\,.\]
\end{corollaire}

The method used to prove the previous theorem does not give an asymptotic for the first $b-1$ non-zero Steklov eigenvalues since the Steklov--Neumann spectrum of a disjoint union of a finite number of domains $D_j$ is the disjoint union of their spectra. 
Indeed, for an eigenpair to satisfy the Steklov--Neumann problem on $D:=\cup D_j$, it must satisfy the Steklov--Neumann problem on each $D_j$, proving that the spectrum of $D$ is included in the disjoint union of the spectra. Conversely, if there is an eigenpair for one of the $D_j$, by prolonging the eigenfunction by $0$ to the other domains, we prove that the eigenvalue is an eigenvalue for $D$.
Thus, the first $b$ Steklov--Neumann eigenvalues are zero. Nonetheless, since the Steklov--Dirichlet eigenvalues are non-zero, we obtain the following upper bound using the method described above.
\begin{theoreme}\label{thm:upperbound}
Let $M$ be a Riemannian manifold of dimension $m\geq 2$. Let $N_1\,,\ldots\,,N_b\subset M$ be $b\geq2$ disjoint closed connected submanifolds such that the dimension of $N_j$ is $0\leq n_j\leq m-2$. Let $n:=\min n_j$. For $1\leq\ell\leq b-1$, we have
\begin{equation}\label{eq:upperbound}
\limsup\limits_{\varepsilon\to0}\varepsilon\sigma_\ell(\Omega_\varepsilon)\leq\max\{1,m-n-2\}\,.
\end{equation}
\end{theoreme}

Using a different method, we also obtain a lower bound for the first non-zero eigenvalue. Namely, we use Poincaré type inequalities to bound below the Dirichlet energy near the boundary of the tubular neighbourhoods.
This method is inspired by the proof of \cite[Proposition 3.3]{CEG19} where the authors estimate the energy cost of the Rayleigh--Steklov quotient on a neighbourhood the boundary of the form $\partial\Omega\times\lbrack\varepsilon,2\varepsilon\rbrack$. However, this choice of interval can not be made here since one would have to suppose that $n\neq m-2$ for the argument to work.
\begin{theoreme}\label{thm:largeegv}
Let $(M,g)$ be a compact Riemannian manifold of dimension $m\geq2$. Let $b\geq2$ be an integer. Let $N_1\,,N_2\,,\ldots\,,N_b\subset M$ be $b$ disjoint closed connected submanifolds such that $N_j$ is of dimension $0\leq n_j\leq m-2$. The first non-zero Steklov eigenvalue of $\Omega_\varepsilon:=M\backslash\bigg(\bigsqcup\limits_{j=1}^b N_j^\varepsilon\bigg)$ satisfies 
\begin{equation}\label{eq:lowerbound}
\liminf\limits_{\varepsilon\to0}\varepsilon^{\frac{1}{m+1}}\sigma_1(\Omega_\varepsilon)\geq C\,,
\end{equation}
where $C$ is an explicit non-zero constant depending on $M$, $b$ and on the submanifolds $N_j$ given by Equation \eqref{eq:constantC}.
\end{theoreme}

The lower bound \eqref{eq:lowerbound} is not sharp. We expect a sharp lower bound to match the upper bound \eqref{eq:upperbound}, as it can be observed in the following example for the $2-$dimensional sphere with two diametrically opposed perforations.
\begin{exemple}
Let $\s^2\subset\R^3$ be the unit sphere and let $N,S\in \s^2$ be the north pole and the south pole respectively. For $0<\varepsilon<\frac{\pi}{2}$, consider the Steklov problem on
\[\Omega_\varepsilon=\s^2\backslash(B_\varepsilon(N)\cup B_\varepsilon(S))\,.\]
Consider the following parametrization on $\Omega_\varepsilon$:
$$\phi(\theta,\psi)=(\cos\theta\sin\psi,\sin\theta\sin\psi,\cos\psi)\,$$
In these coordinates, the normal derivative of a function $u:\Omega_\varepsilon\to\R$ on $\partial B_\varepsilon(N)$ is given by
\[\partial_n u=-u_\psi\,,\]
and on $\partial B_\varepsilon(S)$, the normal derivative is given by
\[\partial_n u=u_\psi\,.\]

Thus, the Steklov problem on $\Omega_\varepsilon$ in these coordinates is
\begin{equation*}
    \begin{cases}
    u_{\psi\psi}+\cot\psi u_\psi+\sec^2\psi u_{\theta\theta}=0&\mbox{ in $\Omega_\varepsilon$}\,,\\
    -u_\psi(\theta,\varepsilon)=\sigma u(\theta,\varepsilon)\,,\\
    u_\psi(\theta,\pi-\varepsilon)=\sigma u(\theta,\pi-\varepsilon)\,.
    \end{cases}
\end{equation*}
Let us use the method of separation of variables, i.e. suppose that $u(\theta,\psi)=F(\theta)G(\psi)$. We obtain the following problems:
\begin{equation}\label{eq:prob1}
    \begin{cases}
        -F''=\lambda F \,,\\
        F(0)=F(2\pi)\,.
    \end{cases}
\end{equation}
\begin{equation}\label{eq:prob2}
    \begin{cases}
    \sin^2\psi G''+\sin\psi\cos\psi G'-\lambda G=0\,,\\
    -G'(\varepsilon)=\sigma G(\varepsilon)\,,\\
    G'(\pi-\varepsilon)=\sigma G(\pi-\varepsilon)\,.
    \end{cases}
\end{equation}
Problem \eqref{eq:prob1} has a solution for $\lambda=n^2$, $n\in\N$.
Then, let us solve \eqref{eq:prob2} for $\lambda_n=n^2$.
When $n=0$, the two independent solutions of the differential equation are
\begin{gather*}
    G_0^1(\psi)=1\,,\\
    G_0^2(\psi)=\log(\tan(\frac{\psi}{2}))\,.
\end{gather*}
The function $G_0^1$ has the associated eigenvalue $\sigma_0=0$. The function $G_0^2$ has the associated eigenvalue
\[\sigma=\frac{1}{\sin\varepsilon\log(\cot(\frac{\varepsilon}{2}))}\,.\]
When $n\geq1$, the solution to the differential equation is
\[G(\psi)=c_1\cot^n\bigg(\frac{\psi}{2}\bigg)+c_2\tan^n\bigg(\frac{\psi}{2}\bigg)\,.\]
In order for the constants $c_1$ et $c_2$ to satisfy the limiting conditions, they must satisfy the following matrix equation:
$$
\begin{pmatrix}
\frac{n\cot^{n-1}(\frac{\varepsilon}{2})\csc^2(\frac{\varepsilon}{2})}{2}-\sigma\cot^n(\frac{\varepsilon}{2})&\frac{-n\tan^{n-1}(\frac{\varepsilon}{2})\sec^2(\frac{\varepsilon}{2})}{2}-\sigma\tan^n(\frac{\varepsilon}{2})\\
\frac{-n\tan^{n-1}(\frac{\varepsilon}{2})\sec^2(\frac{\varepsilon}{2})}{2}-\sigma\tan^n(\frac{\varepsilon}{2})&\frac{n\cot^{n-1}(\frac{\varepsilon}{2})\csc^2(\frac{\varepsilon}{2})}{2}-\sigma\cot^n(\frac{\varepsilon}{2})
\end{pmatrix}
\begin{pmatrix}
c_1\\c_2
\end{pmatrix}=0\,.
$$
This equation admits a non trivial solution if the matrix is non invertible, that is if its determinant is zero:
\begin{multline*}
    \sigma^2\bigg(\cot^{2n}\bigg(\frac{\varepsilon}{2}\bigg)-\tan^{2n}\bigg(\frac{\varepsilon}{2}\bigg)\bigg)-2n\csc\varepsilon\sigma \bigg(\cot^{2n}\bigg(\frac{\varepsilon}{2}\bigg)+\tan^{2n}\bigg(\frac{\varepsilon}{2}\bigg)\bigg)\\
    +n^2\csc^2\varepsilon\bigg(\cot^{2n}\bigg(\frac{\varepsilon}{2}\bigg)-\tan^{2n}\bigg(\frac{\varepsilon}{2}\bigg)\bigg)=0\,.
\end{multline*}
This gives the following eigenvalues
\begin{gather*}
    \sigma_n^+=\frac{n(1+\tan^{2n}(\frac{\varepsilon}{2}))}{\sin\varepsilon(1-\tan^{2n}(\frac{\varepsilon}{2}))}\,,\\
    \sigma_n^-=\frac{n(1-\tan^{2n}(\frac{\varepsilon}{2}))}{\sin\varepsilon(1+\tan^{2n}(\frac{\varepsilon}{2}))}\,.
\end{gather*}
By the series expansion of $\sin(x)$ and of $\tan^{2n}(\frac{x}{2})$ at $x=0$, when $\varepsilon\to0$, the behaviour of the eigenvalues $\sigma_n^+$ and $\sigma_n^-$ is
\[\frac{n}{\varepsilon}+o(\varepsilon^{-k})\,,\]
for all $k>1$. 
\end{exemple}

\subsection{Plan of the paper}
The proofs of the main results are based on quasi-isometries, which are introduced in Section \ref{sec:quasiiso}. 
The proof of Theorems \ref{thm:convergence} and \ref{thm:upperbound} are discussed in Section \ref{sec:upperbound} and Theorem \ref{thm:largeegv} is demonstrated in Section \ref{sec:largeegv}.

\subsection{Notation and convention}
Here is a list of notations and conventions that are used throughout the paper.

\begin{itemize}
\item For a Riemannian manifold $M$, $g$ is its Riemannian metric and $m$ is its dimension. Moreover, $m\geq2$.
\item For a submanifold $N\subset M$, $n$ is its dimension, with $0\leq n\leq m-2$ and $h$ is the restriction of $g$ to $N$, with the convention that $h\equiv0$ if $N$ is a point. Moreover, $d=m-n-1$.
\item $|W|_g$ is the volume of the domain $W\subset M$ with respect to the metric $g$, with the convention that $|W|=1$ if $W$ is a point.
\item $g_0$ is the round metric on the $d-$dimensional sphere $\s^d$ and $\omega_d=|\s^d|_{g_0}$.
\end{itemize}

\section{Quasi-isometry}\label{sec:quasiiso}
In this section, we recall some useful results about quasi-isometry which are used to prove the main theorems.

\begin{definition}\label{def:quasiiso}
Let $g_1\,,g_2$ be two Riemannian metrics on a given manifold $M$. We say that $g_1$ and $g_2$ are quasi-isometric with constant $K\geq 1$ if for all $p\in M$ and for all $v\in T_pM\backslash\{0\}$,
\[\frac{1}{K}\leq\frac{g_1(v,v)}{g_2(v,v)}\leq K\,.\]
\end{definition}

Let $p\in (M,g)$. Using the expression of the metric in normal coordinates, it can be shown that a small ball around $p$ is quasi-isometric to a Euclidean ball of the same radius. 
Replacing $p$ by a submanifold $N\subset M$ of positive codimension, a similar result can be shown by using the expression of the metric in Fermi coordinates given as follow.
Given a point $p_0\in N$, there exists a system of coordinates $(y_1,\ldots,y_n)$ on a open neighbourhood $U\subset N$ containing $p_0$ and a small neighbourhood $\mathcal{O}\subset \{(q,v)~|~ q\in U\,,v\perp T_qN\}$ such that the exponential map $\exp|_\mathcal{O}$ is a diffeomorphism from $\mathcal{O}$ onto its image.
The Fermi coordinates around $p_0$ on $\exp(\mathcal{O})\subset M$ are given by $(x_1,\ldots,x_m):=(y_1,\ldots,y_n,\exp_{(y_1,\ldots,y_n)}^{-1})$. See \cite{tubes} for a concise presentation of these coordinates and their fundamental properties.
On $N\cap \exp(\mathcal{O})$, the functions $x_1,\ldots,x_n$ form a system of coordinates on $N$. Moreover, on $N\cap \exp(\mathcal{O})$, the vector fields $\frac{\partial}{\partial x_i}$, for $i=n+1,\ldots,m$, are orthonormal. 
Thus, for every $p\in N\cap \exp(\mathcal{O})$, the metric $g$ is of the form
\begin{gather*}
 g_{ij}(p)=
  \begin{cases}
    h_{ij}&\mbox{ for $1\leq i,j\leq n$}\,, \\
    0&\mbox{ for $1\leq i \leq n$ and $n+1\leq j\leq m$}\,,\\
    \delta_{ij}&\mbox{ for $n+1\leq i,j\leq m$}\,.
  \end{cases}
\end{gather*}
The metric expressed in this form is used in the proof of the next proposition which shows that a small tubular neighbourhood around $N$ is quasi-isometric to a cylinder.

\begin{proposition}\label{prop:quasiiso}
Let $(M,g)$ be an $m-$dimensional Riemannian compact manifold and $N\subset M$ a compact submanifold of dimension $n$. Let $g_E$ be the $(m-n)-$dimensional Euclidean metric. For every $\varepsilon_0>0$, there exist $\delta>0$ such that, on $N^\delta:=\{p\in M~|~\text{dist}(p,N)<\delta\}$, $g$ is quasi-isometric to the product metric $\Tilde{g}=h\oplus g_E$ with constant $1+\varepsilon_0$. 
\end{proposition}

See \cite[Proposition 2.2]{brisson} for a detailed proof. Throughout the article, the decomposition of the Euclidean metric in the polar metric $g_E=dr^2\oplus r^2 g_0$, with $g_0$ being the round metric on the sphere, is used.
The next proposition shows that if two metrics are quasi-isometric, then their Steklov eigenvalues are relatively close one to another, see \cite[Proposition 2.2]{CEG19}. The result is only stated for Steklov eigenvalues, but it is also true for Steklov--Neumann and Steklov--Dirichlet eigenvalues.

\begin{proposition}\label{prop:quasiisovp}
Let $M$ be a Riemannian manifold of dimension $m$. Let $g_1\,,g_2$ be two Riemannian metrics on $M$ which are quasi-isometric with constant $K$. The Steklov eigenvalues with respect to $g_1$ and to $g_2$ satisfy the following inequality
\[\frac{1}{K^{m+1/2}}\leq\frac{\sigma_\ell(M,g_1)}{\sigma_\ell(M,g_2)}\leq K^{m+1/2}\,.\]
\end{proposition}

\section{Upper bound for the first $b-1$ non-zero eigenvalues and rate of divergence for the truncated spectrum}\label{sec:upperbound}

In this section, Theorem \ref{thm:upperbound} and Theorem \ref{thm:convergence} are proved. The proofs use the Dirichlet--Neumann bracketing of the Steklov eigenvalues and a key lemma from \cite{brisson}.

Let $\Omega$ be a Riemannian manifold with boundary. Consider $A\subsetneq\Omega$  a neighborhood of the boundary $\partial\Omega$. Let $\Sigma:=\partial A\setminus\partial\Omega$ be the inner part of the boundary of $A$.
We introduce the following mixed Steklov--Neumann and Steklov--Dirichlet:
\begin{equation*}
    \begin{cases}
    \Delta u=0&\mbox{ in $A$}\,,\\
    \partial_n u=0 &\mbox{ on $\Sigma$}\,,\\
    \partial_n u=\sigma^{SN} u &\mbox{ on $\partial\Omega$}\,,
  \end{cases}
  \qquad\qquad\text{ and }\qquad\qquad
    \begin{cases}
    \Delta u=0&\mbox{ in $A$}\,,\\
    u=0 &\mbox{ on $\Sigma$}\,,\\
    \partial_n u=\sigma^{SD} u &\mbox{ on $\partial\Omega$}\,.
    \end{cases}
\end{equation*}
Their spectra are given by unbounded sequences of eigenvalues
$$0=\sigma_0^{SN}(A)\leq\sigma_1^{SN}(A)\leq\sigma_2^{SN}(A)\leq\cdots$$
and
$$0<\sigma_1^{SD}(A)\leq\sigma_2^{SD}(A)\leq\sigma_3^{SD}(A)\leq\cdots$$
where each eigenvalue is repeated according to its multiplicity. 
It follows from their variational characterizations that for all $\ell\geq0$, the following inequality holds:
\[\sigma_\ell^{SN}(A)\leq\sigma_\ell(\Omega)\leq\sigma_{\ell+1}^{SD}(A)\,.\]
This is a classical application of the Dirichlet--Neumann bracketing. See  \cite[Section 2]{colbois2018steklov} for details.

The following lemma is key in the proofs of Theorems \ref{thm:upperbound} and \ref{thm:convergence}. It is based on the results obtained by separation of variables in \cite[Lemma 3.3, Lemma 3.6, Theorem 4.2]{brisson}.
\begin{lemme}\label{lemme:asymp}
Let $N$ be a closed connected Riemannian manifold of dimension $0\leq n\leq m-2$ with $m\geq 2$. The distinct Steklov--Dirichlet and Steklov--Neumann eigenvalues on $\lbrack\varepsilon,\delta\rbrack\times N\times\s^{m-n-1}$ equipped with the metric $h \oplus dr^2\oplus r^2g_0$, with the Steklov condition at $\varepsilon$, satisfy the following asymptotics as $\varepsilon\to0$.
\begin{enumerate}
\item If $n=m-2$ and $n>0$, 
\begin{gather*}
    \sigma_{k,0}^{SD}\sim\frac{1}{\varepsilon|\log\varepsilon|}\mbox{ if $k\geq0$}\,,\\
    \sigma_{k,(p)}^{SD}\sim\frac{p}{\varepsilon}\mbox{ in all other cases}\,,
\end{gather*}
\begin{gather*}
    \sigma_{0,0}^{SN}=0\,,\\
    \sigma_{k,0}^{SN}\sim\frac{1}{\varepsilon\bigg(|\log(\sqrt{\lambda_k}\varepsilon)|-\frac{K_0'(\sqrt{\lambda_k}\delta)}{I_0'(\sqrt{\lambda_k}\delta)}\bigg)}\mbox{ if $k\neq0$}\,\\
    \sigma_{k,(p)}^{SN}\sim\frac{p}{\varepsilon}\mbox{ in all other cases}\,.
\end{gather*}
\item If $n\neq m-2$ and $n>0$,
\begin{gather*}
    \sigma_{k,(p)}^{SD}\sim\frac{m-n-2+p}{\varepsilon}\mbox{ for every $k\,,p\geq0$}\,,
\end{gather*}
\begin{gather*}
    \sigma_{0,0}^{SN}=0\,,\\
    \sigma_{k,(p)}^{SN}\sim\frac{m-n-2+p}{\varepsilon} \mbox{ in all other cases}\,.
\end{gather*}
\item If $n=0$, 
\begin{gather*}
\sigma_0^{SD}\sim\frac{1}{\varepsilon|\log\varepsilon|} \mbox{ if $m=2$}\,,\\
\sigma_k^{SD}\sim\frac{m+k-2}{\varepsilon}\mbox{ in all other cases}\,,\\
\sigma_0^{SN}=0\,,\\
\sigma_k^{SN}\sim\frac{m+k-2}{\varepsilon}\mbox{ for $k\geq1$}\,.
\end{gather*}
\end{enumerate}
The multiplicity of the eigenvalues $\sigma_{k,(p)}^{SD}$, $\sigma_{k,(p)}^{SN}$ is $m_p$ and the multiplicity of $\sigma_k^{SD}$ and $\sigma_k^{SN}$ is $m_k$.
\end{lemme}

\begin{proof}[Proof of Theorem \ref{thm:upperbound}]
Let $\varepsilon_0>0$. For each $1\leq j\leq b$, by Proposition \ref{prop:quasiiso}, there exists $\delta_j>0$ such that $g$ and $\Tilde{g}_j:=h_j\oplus dr^2\oplus r^2g_0$ are quasi-isometric of constant $K:=1+\varepsilon_0$ on $N_j^{\delta_j}$. Set $\delta:=\min\bigg\{\delta_j, \frac{\text{dist}_{ij}}{3}\bigg\}$, where $\text{dist}_{ij}=\text{dist}(N_j,N_i)$ for $i\neq j$. Then, $g$ and $\Tilde{g}_j$ are quasi-isometric of constant $1+\varepsilon_0$ on $N_j^\delta$. Let $\Tilde{g}$ be a Riemannian metric such that $\Tilde{g}\bigg|_{N_j^\delta}=\Tilde{g}_j$.

Let $\varepsilon<\delta$.
Consider the Steklov problem on $\Omega_\varepsilon:=M\backslash\bigg(\bigsqcup\limits_{j=1}^b N_j^\varepsilon\bigg)$ equipped with the metric $g$.
By the Dirichlet--Neumann bracketing with $A:=\bigsqcup\limits_{j=1}^b N_j^\delta\backslash N_j^\varepsilon$, we have
\[\sigma_\ell^{SN}(A,g)\leq\sigma_\ell(\Omega_\varepsilon,g)\leq\sigma_{\ell+1}^{SD}(A,g)\,.\]
By Proposition \ref{prop:quasiisovp}, for every $\ell\geq0$, we have
\begin{gather*}
\sigma_\ell^{SN}(A,g)\geq\frac{\sigma_\ell^{SN}(A,\Tilde{g})}{(1+\varepsilon_0)^{2m+1}}\,,\\
\sigma_{\ell+1}^{SD}(A,g)\leq (1+\varepsilon_0)^{2m+1}\sigma_{\ell+1}^{SD}(A,\Tilde{g})\,,
\end{gather*}
which leads to the inequality
\[\frac{\sigma_\ell^{SN}(A,\Tilde{g})}{(1+\varepsilon_0)^{2m+1}}\leq\sigma_\ell(\Omega_\varepsilon,g)\leq(1+\varepsilon_0)^{2m+1}\sigma_{\ell+1}^{SD}(A,\Tilde{g})\,.\]

Since $(A,\Tilde{g})$ is a union of product manifolds, the Steklov--Neumann and Steklov--Dirichlet spectra are the disjoint union of the spectra on each manifolds, that is
\begin{gather*}
\mathcal{S}^{SN}(A,\Tilde{g})=\{\sigma_\ell^{SN}(A,\Tilde{g})\}_{\ell\geq0}=\bigsqcup\limits_{j=1}^b \mathcal{S}^{SN}(N_j^\delta\backslash N_j^\varepsilon,\Tilde{g}_j)\,,\\
\mathcal{S}^{SD}(A,\Tilde{g})=\{\sigma_\ell^{SD}(A,\Tilde{g})\}_{\ell\geq0}=\bigsqcup\limits_{j=1}^b \mathcal{S}^{SD}(N_j^\delta\backslash N_j^\varepsilon,\Tilde{g}_j)\,.
\end{gather*}

Hence, for $0\leq\ell\leq b-1$, we have $\sigma_\ell^{SN}(A,\Tilde{g})=0$. For the Steklov--Dirichlet spectrum, by observing the asymptotics in \ref{lemme:asymp}, if $\varepsilon<1/e$, then the first $b-1$ Steklov--Dirichlet eigenvalues are bounded above by
\[\sigma_\ell^{SD}(A,\Tilde{g})\leq\frac{1}{\varepsilon|\log\varepsilon|}\,,\]
if there exists a $j$ such that $n_j=m-2$ or by
\[\sigma_\ell^{SD}(A,\Tilde{g})\leq\frac{m-n-2}{\varepsilon}\,,\]
otherwise, where $n=\min n_j$.
Since 
\[\frac{1}{\varepsilon|\log\varepsilon|}\leq\frac{1}{\varepsilon}\,,\]
we have that, for every $1\leq\ell\leq b-1$, the $\ell-$th Steklov eigenvalue is bounded above by
\[\sigma_\ell(\Omega_\varepsilon,g)\leq\frac{\max\{1,m-n-2\}(1+\varepsilon_0)^{2m+1}}{\varepsilon}\,.\]
Thus, by taking the limit superior of $\varepsilon\sigma_\ell(\Omega_\varepsilon,g)$ as $\varepsilon\to0$, we obtain that
\[\limsup\limits_{\varepsilon\to0}\varepsilon\sigma_\ell(\Omega_\varepsilon,g)\leq \max\{1,m-n-2\}(1+\varepsilon_0)^{2m+1}\,.\]
Taking the limit as $\varepsilon_0\to0$ in the previous inequality proves the theorem.
\end{proof}

\begin{proof}[Proof of Theorem \ref{thm:convergence}]
By copying the beginning of the last proof, for every $\ell\geq0$, we have
\[\frac{\sigma_\ell^{SN}(A,\Tilde{g})}{(1+\varepsilon_0)^{2m+1}}\leq\sigma_\ell(\Omega_\varepsilon,g)\leq(1+\varepsilon_0)^{2m+1}\sigma_{\ell+1}^{SD}(A,\Tilde{g})\,,\]
where
\begin{gather*}
\{\sigma_\ell^{SN}(A,\Tilde{g})\}_{\ell\geq0}=\bigsqcup\limits_{j=1}^b \mathcal{S}^{SN}(N_j^\delta\backslash N_j^\varepsilon,\Tilde{g}_j)\,,\\
\{\sigma_\ell^{SD}(A,\Tilde{g})\}_{\ell\geq0}=\bigsqcup\limits_{j=1}^b \mathcal{S}^{SD}(N_j^\delta\backslash N_j^\varepsilon,\Tilde{g}_j)\,.
\end{gather*}

Thus, for every $\ell\geq b$, there exists $1\leq j\leq b$ such that $\sigma_\ell^{SD}\in\mathcal{S}^{SN}(N_j^\delta\backslash N_j^\varepsilon,\Tilde{g}_j)$ and $\sigma_\ell^{SN}\in\mathcal{S}^{SD}(N_j^\delta\backslash N_j^\varepsilon,\Tilde{g}_j)$. 

If $n_j\neq0$, there exist indices $k\,,p$ such that $\sigma_\ell^{SD}(A,\Tilde{g})=\sigma_{k,p}^{SD}(N_j^\delta\backslash N_j^\varepsilon,\Tilde{g}_j)$. 
If $p=0$, set $q=0$ and if $p>0$, choose the unique $q>0$ such that $m_0+\cdots+m_{q-1}\leq p< m_0+\cdots+m_{q-1}+m_{q}$. By Lemma \ref{lemme:asymp}, for each $k\geq0$, we have
\[\lim\limits_{\varepsilon\to0}\varepsilon\sigma_{\ell}(\Omega_\varepsilon,g)\leq (1+\varepsilon_0)^{2m+1}(m-n_j-2+q)\,,\]
except where $k=p=0$. 
Since it is true for every $\varepsilon_0>0$, we take the limit as $\varepsilon_0\to0$ to obtain
\[\lim\limits_{\varepsilon\to0}\varepsilon\sigma_{\ell}(\Omega_\varepsilon,g)\leq m-n_j-2+q\,.\]
We also have $\sigma_\ell^{SN}(A,\Tilde{g})=\sigma_{k,p}^{SN}(N_j^\delta\backslash N_j^\varepsilon,\Tilde{g}_j)$. If $p=0$, set $q=0$ and if $p>0$, choose the unique $q>0$ such that $m_0+\cdots+m_{q-1}\leq p< m_0+\cdots+m_{q-1}+m_{q}$.
By Lemma \ref{lemme:asymp}, for each $k\geq0$, we have
\[\lim\limits_{\varepsilon\to0}\varepsilon\sigma_{\ell}(\Omega_\varepsilon,g)\geq\frac{m-n_j-2+q}{(1+\varepsilon_0)^{2m+1}}\,,\]
except when $k=p=0$.
Since it is true for every $\varepsilon_0>0$, we take the limit as $\varepsilon_0\to0$ to obtain
\[\lim\limits_{\varepsilon\to0}\varepsilon\sigma_{\ell}(\Omega_\varepsilon,g)\geq m-n_j-2+q\,,\]
and thus
\[\lim\limits_{\varepsilon\to0}\varepsilon\sigma_{\ell}(\Omega_\varepsilon,g)=m-n_j-2+q\,.\]
This proves Equation \eqref{eq:convergence}.

If $n_j=0$, there exists an index $k$ such that $\sigma_\ell^{SD}\in\mathcal{S}^{SN}(N_j^\delta\backslash N_j^\varepsilon,\Tilde{g}_j)$ and $\sigma_\ell^{SN}\in\mathcal{S}^{SD}(N_j^\delta\backslash N_j^\varepsilon,\Tilde{g}_j)$. 
For $k>0$, choose the unique $q>0$ such that $m_0+\cdots+m_{q-1}\leq k< m_0+\cdots+m_{q-1}+m_{q}$. By Lemma \ref{lemme:asymp}, we have
\[\lim\limits_{\varepsilon\to0}\varepsilon\sigma_{\ell}(\Omega_\varepsilon,g)\leq (1+\varepsilon_0)^{2m+1}(m-2+q)\,.\] 
Since it is true for every $\varepsilon_0>0$, we take the limit as $\varepsilon_0\to0$ to obtain
\[\lim\limits_{\varepsilon\to0}\varepsilon\sigma_{\ell}(\Omega_\varepsilon,g)\leq m-2+q\,.\]
We also have $\sigma_\ell^{SN}(A,\Tilde{g})=\sigma_{k}^{SN}(N_j^\delta\backslash N_j^\varepsilon,\Tilde{g}_j)$. For $k>0$, choose the unique $q>0$ such that $m_0+\cdots+m_{q-1}\leq k< m_0+\cdots+m_{q-1}+m_{q}$.
By Lemma \ref{lemme:asymp}, we have
\[\lim\limits_{\varepsilon\to0}\varepsilon\sigma_{\ell}(\Omega_\varepsilon,g)\geq\frac{m-2+q}{(1+\varepsilon_0)^{2m+1}}\,.\]
Since it is true for every $\varepsilon_0>0$, we take the limit as $\varepsilon_0\to0$ to obtain
\[\lim\limits_{\varepsilon\to0}\varepsilon\sigma_{\ell}(\Omega_\varepsilon,g)\geq m-2+q\,,\]
and thus
\[\lim\limits_{\varepsilon\to0}\varepsilon\sigma_{\ell}(\Omega_\varepsilon,g)= m-2+q\,.\]
This proves Equation \eqref{eq:convergencepoint}.

If $n_j=m-2$ and $p=0$, $k>0$, we can improve the limit. Indeed, we have
\[\lim\limits_{\varepsilon\to0}\varepsilon|\log\varepsilon|\sigma_{\ell}(\Omega_\varepsilon,g)\leq\lim\limits_{\varepsilon\to0}\varepsilon|\log\varepsilon|(1+\varepsilon_0)^{2m+1}\sigma_{j,k,0}^{SD}(N_j^\delta\backslash N_j^\varepsilon,\Tilde{g}_j))\sim (1+\varepsilon_0)^{2m+1}\,.\]
Similarly, we have the lower bound
\[\lim\limits_{\varepsilon\to0}\varepsilon|\log\varepsilon|\sigma_{\ell}(\Omega_\varepsilon,g)\geq\lim\limits_{\varepsilon\to0}\frac{\varepsilon|\log\varepsilon|}{(1+\varepsilon_0)^{2m+1}}\sigma_{j,k,0}^{SN}(N_j^\delta\backslash N_j^\varepsilon,\Tilde{g}_j)\sim\frac{1}{(1+\varepsilon_0)^{2m+1}}\,.\]
Since it is true for every $\varepsilon_0>0$, we take the limit as $\varepsilon_0\to0$ to obtain
\[\lim\limits_{\varepsilon\to0}\varepsilon|\log\varepsilon|\sigma_{\ell}(\Omega_\varepsilon,g)=1\,.\]
\end{proof}

\section{Large first Steklov eigenvalue for manifolds with multiple tubular excisions}\label{sec:largeegv}

In this section, we prove Theorem \ref{thm:largeegv}.
Recall that in Proposition \ref{prop:quasiiso}, it is proved that for $N\subset M$ a closed connected submanifold of dimension $0\leq n\leq m-2$, a tubular annulus around $N$ of inner radius $\varepsilon$ and outer radius $\delta$ is quasi-isometric to $N\times\B^{m-n}(\delta)\backslash\B^{m-n}(\varepsilon)$.
This quasi-isometry is used in the proof of Theorem \ref{thm:upperbound} to construct a metric $\Tilde{g}$ quasi-isometric to $g$ on $M$. The only information on $\Tilde{g}$ needed is its value on a neighbourhood of the submanifolds, meaning that the value of $\Tilde{g}$ outside of those neighbourhoods is not controlled. 
In the proof of Theorem \ref{thm:largeegv}, the value of the metric outside of those neighbourhoods needs to be controlled. We give the construction here.

Let $\varepsilon_0>0$. For each $j\in\{1,\,,\ldots\,,b\}$, by Proposition \ref{prop:quasiiso}, there exists $\delta_j>0$ such that $g$ and $\Tilde{g}_j:=h_j\oplus dr^2\oplus r^2g_0$, where $h_j$ is the metric on $N_j$, are quasi-isometric of constant $K:=1+\varepsilon_0$ on $N_j^{\delta_j}$. Set $\delta:=\min\bigg\{\frac{\delta_j}{3}, \frac{\text{dist}_{ij}}{6}\bigg\}$, where $\text{dist}_{ij}=\text{dist}(N_j,N_i)$ for $i\neq j$. Then, $g$ and $\Tilde{g}_j$ are quasi-isometric on $N_j^\delta$.
Let $\chi\in C^\infty(M)$ be a cut-off function such that
\begin{gather*}
0\leq\chi\leq1\,,\\
\chi\equiv 1 \text{ on } \bigsqcup\limits_{j=1}^b N_j^\delta \,,\\
\chi\equiv0 \text{ on } M\backslash \bigsqcup\limits_{j=1}^b N_j^{3\delta/2}\,.
\end{gather*}
Define $\overline{g}$ a new Riemannian metric on $M$ as
\begin{equation*}
\overline{g}:=\begin{cases}
(1-\chi)g+\chi\Tilde{g}_j\,,&\mbox{ on $N_j^{3\delta/2}$ for all $1\leq j\leq b$,}\\
g\,,&\mbox{ elsewhere.}
\end{cases}
\end{equation*}
Then, let $x\in M$. If $x\in N_j^{\delta}$ for a certain $j$, then we have for every $v\in T_x M\backslash\{0\}$
\[\frac{\overline{g}(v,v)}{g(v,v)}=\frac{\Tilde{g}_j(v,v)}{g(v,v)}\,,\]
which implies that
\[\frac{1}{1+\varepsilon_0}\leq\frac{\overline{g}(v,v)}{g(v,v)}\leq 1+\varepsilon_0\,,\]
since $g$ and $\Tilde{g}_j$ are quasi-isometric of constant $1+\varepsilon_0$.
If $x\in N_j^{3\delta/2}\backslash N_j^\delta$ for a certain $j$, then, since $3\delta/2\leq\delta_j/2$, it follows that $g$ and $\Tilde{g}_j$ are quasi-isometric of constant $1+\varepsilon_0$ on $N_j^{3\delta/2}\backslash N_j^\delta$. Thus, for every $v\in T_x M\backslash\{0\}$, we have
\[\frac{1}{1+\varepsilon_0}\leq 1-\chi(x)+\frac{\chi(x)}{1+\varepsilon_0}\leq\frac{\overline{g}(v,v)}{g(v,v)}\leq 1-\chi(x)+\chi(x)(1+\varepsilon_0)\leq 1+\varepsilon_0\,.\]
If $x\in M\backslash \bigsqcup\limits_{j=1}^b N_j^{3\delta/2}$, then we have $\overline{g}=g$.
We conclude that $g$ and $\overline{g}$ are quasi-isometric on $M$ with constant $1+\varepsilon_0$.

We start by giving a lower bound for the Dirichlet energy of a function defined on $T(\varepsilon,\delta):=N\times(\B^{m-n}(\delta)\backslash\B^{m-n}(\varepsilon))$ equipped with the metric
\[\Tilde{g}:=h\oplus dr^2\oplus r^2g_0\,.\]
\begin{lemme}\label{prop:dirichletenergy}
Let $\sigma_1^{SN}$ be the first non-zero Steklov--Neumann eigenvalue of $T(\varepsilon,\delta)$, with Steklov condition on the boundary component $N\times\s^{d}_\varepsilon$. For $f\in C^\infty(T(\varepsilon,\delta))$, its Dirichlet energy has the following lower bound
\begin{equation*}
\|\nabla f\|^2_{L^2(T(\varepsilon,\delta))}\geq \sigma_1^{SN}\|f-\overline{f}\|^2_{L^2(N\times\s^d_\varepsilon)}\,,
\end{equation*}
where $\overline{f}:=\frac{1}{|N||\s^d_\varepsilon|}\int\limits_{N\times\s^d_\varepsilon} f dS_{\Tilde{g}}$.
\end{lemme}

\begin{proof}
The eigenvalue $\sigma_1^{SN}$ admits the following characterization
\[\sigma_1^{SN}=\min\bigg\{\frac{\|\nabla f\|^2_{L^2(T(\varepsilon,\delta))}}{\|f\|^2_{L^2(N\times\s^{d}_\varepsilon)}}~:~ \int\limits_{N\times\s^{d}_\varepsilon} f\,dS_{\Tilde{g}}=0\bigg\}\,.\]

Define $\overline{f}:=\frac{1}{|N||\s^{d}_\varepsilon|}\int\limits_{N\times\s^{d}_\varepsilon} f \,dS_{\Tilde{g}}$ and $F:=f-\overline{f}$. Thus, $F$ satisfies
\[\int\limits_{N\times\s^{d}_\varepsilon} F\,dS_{\Tilde{g}}=0\,.\]
Hence, we have that
\[\sigma_1^{SN}\leq \frac{\int\limits_{T(\varepsilon,\delta)}|\nabla F|^2\,dv_{\Tilde{g}}}{\int\limits_{N\times\s^{d}_\varepsilon}F^2\,dS_{\Tilde{g}}}=\frac{\int\limits_{T(\varepsilon,\delta)}|\nabla f|^2\,dv_{\Tilde{g}}}{\int\limits_{N\times\s^{d}_\varepsilon}F^2\,dS_{\Tilde{g}}}\,.\]
\end{proof}

We also need a Poincaré type inequality for Radon measure on a compact manifold $(\Omega,g)$ based on \cite[Lemma 3.4]{CEG19}. As in \cite{gkl}, we define the first variational eigenvalue of $\Omega$ associated to $\mu$ as
\[\lambda_1(\Omega,g,\mu)=\inf\limits_{F_2}\sup\limits_{f\in F\backslash\{0\}}\frac{\int\limits_\Omega |\nabla f|^2\,dv_g}{\int\limits_\Omega f^2\,d\mu}\,,\]
where the infimum is taken over all $2-$dimensional subspaces $F_2\in C^\infty(\Omega)$.
\begin{lemme}\label{lemme:poincare}
Let $(\Omega,g)$ be a compact connected Riemannian manifold and let $\mu$ be a Radon measure on $\Omega$ not supported on a single point and such that $\mu(\Omega)<\infty$. Let $\lambda_1(\Omega,g,\mu)$ be the first variational eigenvalue of $\Omega$ associated to $\mu$. Let $V_1$ and $V_2$ be two disjoint measurable subsets of $\Omega$ such that $\mu(V_1)\,,\mu(V_2)\neq0$. Then, every function $f\in C^\infty(\Omega)$ satisfies
\[\int\limits_\Omega |\nabla f|^2\,d\mu\geq \frac{\lambda_1(\Omega,g,\mu)}{2}\min(\mu(V_1),\mu(V_2))\bigg(\fint\limits_{V_1}f\,d\mu-\fint\limits_{V_2}f\,d\mu\bigg)^2\,,\]
where $\fint\limits_{V_i}f\,dv_g=\frac{1}{\mu(V_i)}\int\limits_{V_i}f\,d\mu$.
\end{lemme}

We refer to \cite[Lemma 3.4]{CEG19} for a proof of this result.


This lemma is used in the proof of Theorem \ref{thm:largeegv} with the measure $\mu_\varepsilon=dv_{\bar{g}}^\varepsilon$, the measure volume associated to $\bar{g}$ restricted to $\Omega_\varepsilon$. In that case, $\lambda_1(\varepsilon)=\lambda_1(\Omega_\varepsilon,\bar{g},\mu_\varepsilon)$ is the first non-zero Neumann eigenvalue of $\Omega_\varepsilon$ (see \cite[Example 4.2]{gkl}).
If $\mu=dv_{\bar{g}}$ denotes the measure volume associated to $\bar{g}$ on $M$, then the following result shows that $\lambda_1(\varepsilon)$ converges to the first Laplace eigenvalue $\lambda_1(M,\bar{g},\mu)=\lambda_1$ of $M$ as $\varepsilon\to0$.
\begin{proposition}\label{prop:neumannev}
Let $(M,g)$ be a compact Riemannian manifold of dimension $m\geq 2$. Let $N_1\,,\ldots\,,N_b\subset M$ be $b\geq1$ closed disjoint connected submanifolds of dimension $0\leq n_j\leq m-2$. For $\varepsilon>0$, consider $\Omega_\varepsilon:=M\backslash\bigsqcup\limits_{j=1}^b N_j^\varepsilon$. Let $\lambda_1(\varepsilon)$ be the first non-zero Neumann eigenvalue of $\Omega_\varepsilon$ and $\lambda_1$ be the first non-zero Laplace eigenvalue of $M$. Then, we have
\[\lim\limits_{\varepsilon\to0}\lambda_1(\varepsilon)=\lambda_1\,.\]
\end{proposition}
This proposition is a generalization to multiple tubular excisions of \cite[Theorem 2]{anne1987} which states that, for a compact Riemannian manifold $M$ and a submanifolofd $N$ of codimension $\geq2$, the Neumann eigenvalues of $M\backslash N^\varepsilon$ converge to the Laplace eigenvalues of $M$ as $\varepsilon\to 0$. The proof uses the same argument, thus we refer the readers to \cite{anne1987} and references therein for a proof.

We are now ready to prove Theorem \ref{thm:largeegv}.
\begin{proof}[Proof of Theorem \ref{thm:largeegv}]
Let $\varepsilon_0>0$. As seen at the beginning of this section, $g$ and $\overline{g}$ are quasi-isometric on $M$ with constant $1+\varepsilon_0$.

Let $\varepsilon<\min\{\delta,1\}$ be such that $\varepsilon^{\alpha}<\delta$ with $\alpha<1$ to be chosen accordingly later in the proof.
Consider the Steklov problem on $\Omega_\varepsilon:=M\backslash\bigg(\bigsqcup\limits_{j=1}^b N_j^\varepsilon\bigg)$ equipped with the metric $g$.
By quasi-isometry between $\overline{g}$ and $g$, it follows by Proposition \ref{prop:quasiisovp} that
\[\sigma_1(\Omega_\varepsilon,g)\geq \frac{1}{(1+\varepsilon_0)^{m+1/2}}\sigma_1(\Omega_\varepsilon,\overline{g})\,.\]
Recall $d_j=m-n_j-1$ where $n_j$ is the dimension of $N_j$.
Let $f:\Omega_\varepsilon\to\R$ be a smooth function such that
\begin{gather}
    0=\int\limits_{\partial\Omega_\varepsilon}f\,dS_{\overline{g}}=\sum\limits_{j=1}^b \int\limits_{N_j\times\s^{d_j}_\varepsilon}f\,dS_{\Tilde{g}_j}\,,\label{eq:condition1}\\
    \|f\|_{L^2(\partial\Omega_\varepsilon,\overline{g})}=1\label{eq:condition2}\,.
\end{gather}
Let $T_j(\varepsilon,\delta):=N_j\times\lbrack\varepsilon,\delta\rbrack\times\s^{d_j}$, the Rayleigh--Steklov quotient of $f$ in the metric $\overline{g}$ is bounded below by
\[R(f,\overline{g})=\int\limits_{\Omega_\varepsilon}|\nabla f|^2\,dv_{\overline{g}}\geq \sum\limits_{j=1}^b\int\limits_{T_j(\varepsilon,\delta)}|\nabla f|^2\,dv_{\Tilde{g}_j}\,.\]
By applying Lemma \ref{prop:dirichletenergy} to each $N_j$, we have
\[R(f,\overline{g})\geq \min\limits_j \sigma_1^{SN}(T_j(\varepsilon,\delta))\sum\limits_{j=1}^b \bigg\|f-\fint\limits_{N_j\times\s_\varepsilon^{d_j}}f\bigg\|^2_{L^2(N_j\times\s^{d_j}_\varepsilon)}\,.\]
By the parallelogram law and by condition \eqref{eq:condition2}, it follows that
\[R(f,\overline{g})\geq\min\limits_j \sigma_1^{SN}(T_j(\varepsilon,\delta))\bigg(\frac{1}{2}-\sum\limits_{j=1}^b|N_j||\s_\varepsilon^{d_j}|\bigg(\fint\limits_{N_j\times\s_\varepsilon^{d_j}}f\,dS_{\Tilde{g}}\bigg)^2\bigg)\,.\]

We now consider two cases:
\[\sum\limits_{j=1}^b |N_j||\s^{d_j}_\varepsilon|\bigg(\fint\limits_{N_j\times\s^{d_j}_\varepsilon}f\,dS_{\Tilde{g}_j}\bigg)^2\leq\frac{1}{4}\,,\]
and
\[\sum\limits_{j=1}^b |N_j||\s^{d_j}_\varepsilon|\bigg(\fint\limits_{N_j\times\s^{d_j}_\varepsilon}f\,dS_{\Tilde{g}_j}\bigg)^2\geq\frac{1}{4}\,.\]
If $\sum\limits_{j=1}^b |N_j||\s^{d_j}_\varepsilon|\bigg(\fint\limits_{N_j\times\s^{d_j}_\varepsilon}f\,dS_{\Tilde{g}_j}\bigg)^2\leq\frac{1}{4}$, then the Rayleigh--Steklov quotient is bounded below by
\begin{gather*}
R(f,\overline{g})\geq \frac{\min\limits_j \sigma_1^{SN}(T_j(\varepsilon,\delta))}{4}\,.
\end{gather*}
By Lemma \ref{lemme:asymp} part $(2)$, the first positive Steklov--Neumann eigenvalue of $T_j(\varepsilon,\delta)$, as $\varepsilon\to0$, behaves as 
$$
\sigma_1^{SN}(T_j(\varepsilon,\delta))\sim\begin{cases}
\frac{m-n_j-1}{\varepsilon}\,,& \mbox{ if $n_j\neq m-2>0$,}\\
\frac{1}{\varepsilon}\,,&\mbox{ if $n_j=m-2>0$,}\\
\frac{m-1}{\varepsilon}\,,&\mbox{ if $n_j=0$.}
\end{cases}
$$
Since $m-1\geq m-2$, it follows that there exists a $\varepsilon_1>0$ such that, for all $\varepsilon<\min\{\delta,1,\varepsilon_1\}$, we have 
$$
\min\limits_j \sigma_1^{SN}(T_j(\varepsilon,\delta))\geq\frac{\max\{\min_j m-n_j-2,1\}}{\varepsilon}\,.
$$
Since $\varepsilon\leq\varepsilon^\alpha$, because $\alpha<1$, we have that
\[\frac{\max\{\min_j m-n_j-2,1\}}{\varepsilon}\geq \frac{\max\{\min_j m-n_j-2,1\}}{\varepsilon^{\alpha}}\,,\]
and thus
\[\min\limits_j \sigma_1^{SN}(T_j(\varepsilon,\delta))\geq \frac{\max\{\min_j m-n_j-2,1\}}{\varepsilon^{\alpha}}\,.\]
Then, it follows that
\begin{equation}\label{eq:case1}
R(f,\overline{g})\geq\frac{\max\{\min_j m-n_j-2,1\}}{4\varepsilon^{\alpha}}\,.
\end{equation}

If $\sum\limits_{j=1}^b |N_j||\s^{d_j}_\varepsilon|\bigg(\fint\limits_{N_j\times\s^{d_j}_\varepsilon}f\,dS_{\Tilde{g}_j}\bigg)^2\geq\frac{1}{4}$, without loss of generality, suppose that 
$$|N_1||\s^{d_1}_\varepsilon|\bigg(\fint\limits_{N_1\times\s^{d_1}_\varepsilon}f\,dS_{\Tilde{g}_1}\bigg)^2\geq\frac{1}{4b}\,,$$
which implies that
\[\fint\limits_{N_1\times\s^{d_1}_\varepsilon}f\,dS_{\Tilde{g}_1}\geq\frac{1}{\sqrt{4b|N_1||\s^{d_1}_\varepsilon|}}>0\,.\]
By Equation \eqref{eq:condition1}, we have that
\[\sum\limits_{j=2}^b|N_j||\s_\varepsilon^{d_j}|\fint\limits_{N_j\times\s^{d_j}_\varepsilon}f\,dS_{\Tilde{g}_j}=-|N_1||\s_\varepsilon^{d_1}|\fint\limits_{N_1\times\s^{d_1}_\varepsilon}f\,dS_{\Tilde{g}_1}\leq-\sqrt{\frac{|N_1||\s_\varepsilon^{d_1}|}{4b}}\,.\]
Without loss of generality, we can suppose that this implies
\[|N_2||\s_\varepsilon^{d_2}|\fint\limits_{N_2\times\s^{d_2}_\varepsilon}f\,dS_{\Tilde{g}_2}\leq\frac{-1}{b-1}\sqrt{\frac{|N_1|||\s_\varepsilon^{d_1}|}{4b}}\,.\]

Since $d_j\leq m-1$, it follows that $\varepsilon^{d_j}\geq\varepsilon^{m-1}$. Thus, we have
\[\varepsilon^{m-1}|N_2|\omega_{d_2}\fint\limits_{N_2\times\s^{d_2}_\varepsilon}f\,dS_{\Tilde{g}_2}\leq\frac{-1}{b-1}\sqrt{\frac{|N_1||\s_\varepsilon^{d_1}|}{4b}}\leq\frac{-1}{b-1}\sqrt{\frac{\varepsilon^{(m-1)}|N_1|\omega_{d_1}}{4b}}<0\,.\]

We now use the Fourier decomposition of the function $f$ in each $T_j(\varepsilon,\delta)$. Let $\{\phi_i^j\}_{i\geq0}$ be an orthonormal basis of $L^2(N_j\times\s^{d_j})$. Then, the Fourier decomposition of $f$ on $T_j(\varepsilon,\delta)$ is given by
\[f(r,p)=\sum\limits_{i\geq0}a_i^j(r)\phi_i^j(p)\,.\]
Thus, for each $j$, we have
\[\int\limits_{N_j\times\s^{d_j}_\varepsilon}f\,dS_{\Tilde{g}_j}=\int\limits_{N_j\times\s^{d_j}}\sum\limits_{i\geq0}a_0^j(\varepsilon)\phi_i^j\varepsilon^{d_j}\,dv_{h_j\oplus g_0}=\varepsilon^{d_j}a_0^j(\varepsilon)\,.\]
The conditions on the mean value of $f$ in $N_1\times\s_\varepsilon^{d_1}$ and on $N_2\times\s_\varepsilon^{d_2}$ become
\begin{gather*}
a_0^1(\varepsilon)\geq\sqrt{\frac{|N_1|\omega_{d_1}}{4b\varepsilon^{d_1}}}>0\,,\\
a_0^2(\varepsilon)\leq-\frac{1}{(b-1)}\sqrt{\frac{|N_1|\omega_{d_1}}{4b\varepsilon^{m-1}}}<0\,.
\end{gather*} 

If there exists $t_0\in\lbrack\varepsilon,\varepsilon^{\alpha}\rbrack$ such that $a_0^1(t_0)\leq\frac{a_0^1(\varepsilon)}{2}$, then we have 
\[R(f,\overline{g})\geq\int\limits_{T_1(\varepsilon,t_0)}|\nabla f|^2\,dv_{\Tilde{g}_1}\geq\int\limits_{N_1\times\s^{d_1}}\int\limits_\varepsilon^{t_0}(a_0^1(r)')^2r^{d_1}\,dr\,dv_{h_1\oplus g_0}\,.\]
By the Cauchy-Schwarz inequality, it follows that
\[\int\limits_{N_1\times\s^{d_1}}\int\limits_\varepsilon^{t_0}(a_0^1(r)')^2r^{d_1}\,dr\,dv_{h_1\oplus g_0}\geq\frac{\varepsilon^{d_1}}{|N_1|\omega_{d_1}(t_0-\varepsilon)}\bigg(\int\limits_{N_1\times\s^{d_1}}a_0^1(r)'\,dr\,dv_{h_1\oplus g_0}\bigg)^2\,.\]
By the Fundamental theorem of calculus, we have
\[\frac{\varepsilon^{d_1}}{|N_1|\omega_{d_1}(t_0-\varepsilon)}\bigg(\int\limits_{N_1\times\s^{d_1}}a_0^1(r)'\,dr\,dv_{h_1\oplus g_0}\bigg)^2=\frac{\varepsilon^{d_1}}{(t_0-\varepsilon)}\bigg(\int\limits_{N_1\times\s^{d_1}}a_0^1(t_0)-a_0^1(\varepsilon)\,dv_{h_1\oplus g_0}\bigg)^2\geq\frac{\min_j |N_j|^2\omega_{d_j}^2}{16b\varepsilon^{\alpha}}\,.\]
Thus, in that case, the Rayleigh--Steklov quotient is bounded below by
\begin{equation}\label{eq:case2}
R(f,\overline{g})\geq\frac{\min_j |N_j|^2\omega_{d_j}^2}{16b\varepsilon^{\alpha}}\,.
\end{equation}
Similarly, if there exists a $s_0\in\lbrack\varepsilon,\varepsilon^{\alpha}\rbrack$ such that $a_0^2(s_0)\geq\frac{a_0^2(\varepsilon)}{2}$, then 
\begin{equation}\label{eq:case3}
R(f,\overline{g})\geq\frac{\min_j |N_j|^2\omega_{d_j}^2}{16b(b-1)^2\varepsilon^{\alpha}}\,.
\end{equation} 

If for all $t\in\lbrack\varepsilon,\varepsilon^{\alpha}\rbrack$, we have $a_0^1(t)\geq\frac{a_0^1(\varepsilon)}{2}$ and $a_0^2(t)\leq\frac{a_0^2(\varepsilon)}{2}$, then we have
\[\fint\limits_{T_1(\varepsilon,\varepsilon^{\alpha})}f\,dv_{\Tilde{g}_1}=\frac{1}{|T_1(\varepsilon,\varepsilon^{\alpha})|}\int\limits_{N_1\times\s^{d_1}}\int\limits_\varepsilon^{\varepsilon^{\alpha}}a_0^1(r)\phi_0^1r^{d_1}\,dr\,dv_{h_1\oplus g_0}\geq\frac{1}{2}\sqrt{\frac{1}{4b|N_1|\omega_{d_1}\varepsilon^{d_1}}}>0\,,\]
and
\[\fint\limits_{T_2(\varepsilon,\varepsilon^{\alpha})}f\,dv_{\Tilde{g}_2}=\frac{1}{|T_2(\varepsilon,\varepsilon^{\alpha})|}\int\limits_{N_2\times\s^{d_2}}\int\limits_\varepsilon^{\varepsilon^{\alpha}}a_0^2(r)\phi_0^2r^{d_2}\,dr\,dv_{h_2\oplus g_0}\leq\frac{-1}{2|N_2|\omega_{d_2}(b-1)}\sqrt{\frac{|N_1|\omega_{d_1}}{4b\varepsilon^{m-1}}}<0\,,\]
which implies that
\begin{gather*}
\bigg(\fint\limits_{T_1(\varepsilon,\varepsilon^{\alpha})}f\,dS_{\Tilde{g}_1}-\fint\limits_{T_2(\varepsilon,\varepsilon^{\alpha})}f\,dS_{\Tilde{g}_2}\bigg)^2\geq \frac{\min_j|N_j|\omega_{d_j}}{16b\max_j|N_j|^2\omega_{d_j}^2}\bigg(\frac{1}{\varepsilon^{d_1/2}}+\frac{1}{(b-1)\varepsilon^{(m-1)/2}}\bigg)^2\\
\geq \frac{b\min_j|N_j|\omega_{d_j}}{16\varepsilon(b-1)^2\max_j|N_j|^2\omega_{d_j}^2}\,,
\end{gather*}
since $\varepsilon^{d_j}\geq\varepsilon$ for all $j$.

Let $\lambda_1(\varepsilon)$ be the first non-zero Neumann eigenvalue of $\Omega_\varepsilon$ associated to the metric $\overline{g}$ and $\lambda_1$ be the first non-zero eigenvalue of $M$ associated to the metric $\overline{g}$. By Proposition \ref{prop:neumannev}, we have that $\lambda_1(\varepsilon)\to\lambda_1$ as $\varepsilon\to 0$. Thus, there exists $\varepsilon_2>0$ such that, for all $\varepsilon<\min\{\delta,1,\varepsilon_1,\varepsilon_2\}$, we have $\lambda_1(\varepsilon)\geq\frac{\lambda_1}{2}$.
By Lemma \ref{lemme:poincare} with the measure $\mu:=dv_{\overline{g}}$, $V_1=T_1(\varepsilon,\varepsilon^{\alpha})$ and $V_2=T_2(\varepsilon,\varepsilon^{\alpha})$, we have
\begin{gather*}
    R(f,\overline{g})\geq\frac{\lambda_1(\varepsilon)\min_j|T_j(\varepsilon,\varepsilon^{\alpha})|}{2}\bigg(\fint\limits_{T_1(\varepsilon,\varepsilon^{\alpha})}f\,dS_{\Tilde{g}_1}-\fint\limits_{T_2(\varepsilon,\varepsilon^{\alpha})}f\,dS_{\Tilde{g}_2}\bigg)^2\\
   \geq \frac{\lambda_1(\varepsilon)\min_j|T_j(\varepsilon,\varepsilon^{\alpha})|\min_j|N_j|\omega_{d_j}}{32\varepsilon b(b-1)^2\max_j|N_j|^2\omega_{d_j}^2}\geq \frac{\lambda_1(\varepsilon)\varepsilon^{\alpha m}\min_j|N_j|^2\omega_{d_j}^2}{64\varepsilon m b(b-1)^2\max_j|N_j|^2\omega_{d_j}^2}\\
   \geq\frac{\varepsilon^{\alpha m-1}\lambda_1\min_j|N_j|^2\omega_{d_j}^2}{128m b(b-1)^2\max_j|N_j|^2\omega_{d_j}^2}
\end{gather*}
because for each $j$, we have $\varepsilon^{d_j+1}\geq\varepsilon^{m}$. 
In that case, the Rayleigh--Steklov quotient is bounded below by
\begin{equation}\label{eq:case4}
R(f,\overline{g})\geq\varepsilon^{\alpha m-1} \frac{\lambda_1\min_j|N_j|^2\omega_{d_j}^2}{128mb(b-1)^2\max_j|N_j|^2\omega_{d_j}^2}\,.
\end{equation}

In order to compare this lower bound with the ones in \eqref{eq:case1}, \eqref{eq:case2}, \eqref{eq:case3}, we want to find $\alpha<1$ such that $\alpha m-1 =-\alpha$. We the obtain that $\alpha=\frac{1}{m+1}$.

Considering cases \eqref{eq:case1}, \eqref{eq:case2}, \eqref{eq:case3} and \eqref{eq:case4} and the fact that $b-1\geq1$, we obtain that 
\[R(f,\overline{g})\geq\frac{C}{\varepsilon^{\frac{1}{m+1}}}\]
where $C$ is a finite positive constant given by
\begin{equation}\label{eq:constantC}
C:=\min\bigg\{\frac{\max\{\min_j m-n_j-2,1\}}{4}\,,\frac{\min_j |N_j|^2\omega_{d_j}^2}{16b(b-1)^2}\,,\frac{\lambda_1\min_j|N_j|^2\omega_{d_j}^2}{128mb(b-1)^2\max_j|N_j|^2\omega_{d_j}^2}\bigg\}\,.
\end{equation}

By taking the minimum among all functions $f:\Omega_\varepsilon\to\R$ such that $\int\limits_{\Omega_\varepsilon}f\,dS_{\overline{g}}=0$ and $\int\limits_{\Omega_\varepsilon}f^2\,dS_{\overline{g}}=1$, we show by the variational characterization \eqref{eq:varchar2} that
\[\sigma_1(\Omega_\varepsilon,\overline{g})\geq\frac{C}{\varepsilon^{\frac{1}{m+1}}} \,.\]

Thus, we have that
\[\sigma_1(\Omega_\varepsilon,g)\geq\frac{1}{(1+\varepsilon_0)^{m+1/2}}\sigma_1(\Omega_\varepsilon,\overline{g})\geq\frac{1}{(1+\varepsilon_0)^{m+1/2}}\frac{C}{\varepsilon^{\frac{1}{m+1}}}.\]

By taking the limit inferior of $\varepsilon^{\frac{1}{m+1}}\sigma_1(\Omega_\varepsilon,g)$ as $\varepsilon\to0$, we obtain
\[\liminf\limits_{\varepsilon\to0}\varepsilon^{\frac{1}{m+1}}\sigma_1(\Omega_\varepsilon,g)\geq \frac{C}{(1+\varepsilon_0)^{m+1/2}}\,.\]
Taking the limit as $\varepsilon_0\to0$ in the previous inequality proves the theorem.
\end{proof}

\section{Acknowledgements}
The author would like to thank Jean Lagacé for his useful comments and suggestions and Bruno Colbois for reading an early version of the article.
The author was supported by NSERC. This work is a part of the PhD thesis of the author under the supervision of Alexandre Girouard.

\bibliographystyle{plain}
\bibliography{reference_multiple_excision}

\end{document}